\documentclass[reqno,a4paper]{amsart}
\usepackage{amssymb, amsmath, amsthm, latexsym}
\usepackage[latin1]{inputenc}
\usepackage{graphicx}
\usepackage[all]{xy}
\usepackage{xcolor}
\usepackage{hyperref}

\theoremstyle{plain}
\newtheorem{theorem}[subsection]{Theorem}
\newtheorem{proposition}[subsection]{Proposition}
\newtheorem{lemma}[subsection]{Lemma}
\newtheorem{corollary}[subsection]{Corollary}
\theoremstyle{definition}
\newtheorem{definition}[subsection]{Definition}
\theoremstyle{remark}
\newtheorem{remark}[subsection]{Remark}
\newtheorem{example}[subsection]{Example}

\newcommand{\Ima}{\ensuremath{\mathsf{Im}}}
\newcommand{\Ker}{\ensuremath{\mathsf{Ker}}}
\newcommand{\Der}{\ensuremath{\mathsf{Der}}}
\newcommand{\QDer}{\ensuremath{\mathsf{QDer}}}
\newcommand{\IDer}{\ensuremath{\mathsf{IDer}}}
\newcommand{\GenDer}{\ensuremath{\mathsf{GenDer}}}
\newcommand{\End}{\ensuremath{\mathsf{End}}}
\newcommand{\id}{\ensuremath{\mathsf{Id}}}
\newcommand{\Lie}{\ensuremath{\mathsf{Lie}}}
\newcommand{\Leib}{\ensuremath{\mathsf{Leib}}}
\newcommand{\RR}{\ensuremath{\mathsf{RR}}}
\newcommand{\LL}{\ensuremath{\mathsf{LL}}}
\newcommand{\QQ}{\ensuremath{\mathsf{Q}}}
\newcommand{\Hom}{\ensuremath{\mathsf{Hom}}}

\newcommand{\Lieh}{\ensuremath{\mathfrak{h}}}
\newcommand{\Lieg}{\ensuremath{\mathfrak{g}}}
\newcommand{\Liem}{\ensuremath{\mathfrak{m}}}
\newcommand{\Lien}{\ensuremath{\mathfrak{n}}}
\newcommand{\eh}{\ensuremath{\mathfrak{h}}}

\newcommand{\ZLie}{\ensuremath{Z^{\Lie}}}
\newcommand{\K}{\ensuremath{\mathbb{K}}}

\begin{document}
	
\title{On some properties of $\mathsf{Lie}$-centroids of Leibniz algebras}
	
\author{J.M.~Casas}	
\author{X.~Garc\'ia-Mart\'inez}
\author{N.~Pacheco-Rego}
	
\email{jmcasas@uvigo.es}	
\email{xabier.garcia.martinez@uvigo.gal}
\email{natarego@gmail.com}

\address[Jos\'e Manuel Casas]{Universidade de Vigo, Departamento de Matem\'atica Aplicada I, E.\ E.\ Forestal, Campus de Pontevedra, E--36005 Pontevedra, Spain}
\address[Xabier García-Martínez]{Universidade de Vigo, Departamento de Matem\'aticas, Esc.\ Sup.\ de Enx.\ Inform\'atica, Campus de Ourense, E--32004 Ourense, Spain\newline
	and\newline
	Faculty of Engineering, Vrije Universiteit Brussel, Pleinlaan 2, B--1050 Brussel, Belgium}
\address[Natalia Pacheco-Rego]{Instituto Polit\'ecnico do C\'avado e Ave, 4750-810 Vila Frescainha, S. Martinho, Barcelos, Portugal}
	
\thanks{This work was supported by Ministerio de Econom\'ia y Competitividad (Spain), with grant number PID2020-115155GB-I00. The second author is a Postdoctoral Fellow of the Research Foundation--Flanders (FWO)}

\begin{abstract}
	 We study some properties on $\mathsf{Lie}$-centroids related to central $\mathsf{Lie}$-derivations, generalized $\mathsf{Lie}$-derivations and almost inner $\mathsf{Lie}$-derivations. We also determine the $\mathsf{Lie}$-centroid of the tensor product of a commutative associative algebra and a Leibniz algebra.
\end{abstract}



\maketitle


\section{Introduction}

In the semi-abelian categories context, classical commutatory theory deveolped by Higgins and Huq~\cite{Hig, Huq} studies how far are objects of being abelian. Replacing \emph{abelian} by any Birkhoff subcategory, we obtain relative commutator theory. This was first investigated following the lines of relative central extensions of Fröhlich~\cite{Fro} and Janelidze and Kelly's categorical Galois theory~\cite{Jan, JaKe} (see also~\cite{EVDL, Fur, Lue}), and further developed by Everaert and Van der Linden, among others~\cite{CVDL1, EvVdL, Eve}.

The category of Lie algebras forms a Birkhoff subcategory of the category of Leibniz algebras. This means that we can study relative commutator theory o Leibniz algebras respect to Lie, giving rise to interesting developments in the comprehension of both algebraic structures~\cite{BC, BC1, CI, CKh, RC}. Moreover, the interplay between these two categories can be somehow tricky, since it is also possible to find Leibniz algebras as a subcategory of a certain type of Lie algebras~\cite{LoPi2}, although many interesting categorical properties are not preserved~\cite{GaGr, GTVV, GaVa2, GaVa}.

The study of the properties of the centroid of a Lie algebra has been a key step in the classification of finite-dimensional extended affine Lie algebras~\cite{BN, Ne}. For finite-dimensional simple associative algebras, centroids are essential in the investigation of Brauer groups, division algebras and derivations. The theory of centroids in other algebraic structures such as Jordan algebras, superalgebras, $n$-Lie algebras, Zinbiel algebras, among others, can be found, for instance in~\cite{Mc, Ni, Ri, BM, ZL}.

In the recent manuscript~\cite{BCP}, it was introduced the concept of centroid of a Leibniz algebra with respect to the Liesation functor, named $\Lie$-centroid, together with the study of the interplay between $\Lie$-central derivations, $\Lie$-centroids and $\Lie$-stem Leibniz algebras.

Our goal in this paper is to continue with the analysis of properties of the $\Lie$-centroid introduced in~\cite{BCP} and study its interaction with $\Lie$--central derivations, generalised $\Lie$-derivations, quasi-$\Lie$-centroids and almost inner $\Lie$-derivations, establishing a parallelism between the absolute results (classical properties for Lie algebras) with the results obtained. We note that not all classical results are immediately translated into the relative case, so new requirements are needed.

The manuscript is organised as follows: Section~\ref{preliminaries} contains the necessary notions relative to the Liesation functor. In Section~\ref{centroids} we review some properties of $\Lie$--central derivations and $\Lie$-centroids and we obtain some new results on $\Lie$-centroids. In Section~\ref{generalised} we introduce the notions of generalised derivations and quasi-centroids relative to the Liesation functor and we analyse their connections with $\Lie$-centroids.
An important problem in the absolute context tries to determine the conditions under which the central deviations coincide with the almost inner derivations. Our goal in Section~\ref{almost} is the study of this problem in the relative context. The most relevant fact is that we need additional conditions to characterise these conditions. We conclude our study analysing the $\Lie$-centroid of the tensor product of a commutative associative algebra and a Leibniz algebra in Section~\ref{tensor}.

\section{Preliminaries on Leibniz algebras} \label{preliminaries}

Let $\K$ be a fixed ground field such that $\frac{1}{2} \in \K$. Throughout the paper, all vector spaces and tensor products are considered over $\K$.

A \emph{Leibniz algebra}~\cite{Lo2, Lo1} is a vector space ${\Lieg}$ equipped with a bilinear map $[-,-] \colon \Lieg \times \Lieg \to \Lieg$, satisfying the \emph{Leibniz identity}:
\[
 \big[x,[y,z]\big]= \big[[x,y],z\big]-\big[[x,z],y\big], \quad x, y, z \in \Lieg.
\]

A subalgebra ${\eh}$ of a Leibniz algebra ${\Lieg}$ is said to be a \emph{left (resp.~right) ideal} of ${\Lieg}$ if~$[h,g]\in \eh$ (resp.~$ [g,h]\in \eh$), for all $h \in \eh$, $g \in {\Lieg}$. If $\eh$ is both
left and right ideal, then $\eh$ is called a \emph{two-sided ideal} of ${\Lieg}$. In this case $\Lieg/\Lieh$ naturally inherits a Leibniz algebra structure.

Given a Leibniz algebra ${\Lieg}$, we denote by ${\Lieg}^{\textrm{ann}}$ the subspace of ${\Lieg}$ spanned by all elements of the form $[x,x]$, $x \in \Lieg$. It is clear that the quotient ${\Lieg}_{\Lie}=\Lieg/{\Lieg}^{\rm ann}$ is a Lie algebra. This procedure defines the so-called \emph{Liesation functor} $(-)_{\Lie} \colon {\Leib} \to {\Lie}$, which assigns to a Leibniz algebra $\Lieg$ the Lie algebra ${\Lieg}_{\Lie}$. Moreover, the canonical surjective homomorphism ${\Lieg} \twoheadrightarrow {\Lieg}_{\Lie}$ is universal among all homomorphisms from $\Lieg$ to a Lie algebra, implying that the Liesation functor is left adjoint to the inclusion functor $ {\Lie} \hookrightarrow {\Leib}$.

Given a Leibniz algebra $\Lieg$, we define the bracket
\[
[-,-]_{\Lie}\colon\Lieg\to \Lieg, ~~ \text{by} ~~ [x,y]_{\Lie}=[x,y]+[y,x], ~~ \text{for} ~~ x,y\in\Lieg.
\]

Let ${\Liem}$, ${\Lien}$ be two-sided ideals of a Leibniz algebra ${\Lieg}$. The following notions come from~\cite{CKh}, which were derived from~\cite{CVDL1}.

 The \emph{$\Lie$-commutator} of ${\Liem}$ and ${\Lien}$ is the two-sided ideal of $\Lieg$
\[
[\Liem,\Lien]_{\Lie}= \langle \{[m,n]_{\Lie}, m \in \Liem, n \in \Lien \}\rangle.
\]

The \emph{$\Lie$-centre} of the Leibniz algebra $\Lieg$ is the two-sided ideal
\[
\ZLie(\Lieg) = \{ z\in \Lieg \mid \text{$[x,z]_{\Lie}=0$ for all $x\in \Lieg$}\}.
\]

 The \emph{\Lie-centraliser} of ${\Liem}$ and ${\Lien}$ over ${\Lieg}$ is
\[
C_{\Lieg}^{\Lie}({\Liem} , {\Lien}) = \{x \in {\Lieg} \mid [x, m]_{\Lie} \in {\Lien}, \; \text{for all} \;
m \in {\Liem} \} \; .
\]
When $\Lien= 0$, we denote it by $C_{\Lieg}^{\Lie}(\Liem)$. Obviously, $C_{\Lieg}^{\Lie}(\Lieg)=\ZLie(\Lieg)$.

The \emph{right-centre} of $\Lieg$ is the two-sided ideal
\[
	Z^r(\Lieg) = \{ z \in \Lieg\mid [x,z]=0 ~\text{for all}~x \in \Lieg \},
\]
whereas the \emph{left-centre} of a Leibniz algebra $\Lieg$ is the set
\[
Z^l(\Lieg) = \{ z \in \Lieg\mid [z,x]=0 ~\text{for all}~x \in \Lieg \},
\]
which might not even be a subalgebra. The \emph{centre} of $\Lieg$ is the two-sided ideal obtained by $Z(\Lieg) = Z^l(\Lieg) \cap Z^r(\Lieg)$.

Let ${\Lien}$ be a two-sided ideal of a Leibniz algebra $\Lieg$. The \emph{lower $\Lie$-central series} of~$\Lieg$ relative to ${\Lien}$ is the sequence
\[
\cdots \trianglelefteq {\gamma_i^{\Lie}(\Lieg,\Lien)} \trianglelefteq \cdots \trianglelefteq \gamma_2^{\Lie}(\Lieg,\Lien) \trianglelefteq {\gamma_1^{\Lie}(\Lieg,\Lien)}
\]
of two-sided ideals of $\Lieg$ defined inductively by
\[
{\gamma_1^{\Lie}(\Lieg,\Lien)} = {\Lien} \quad \text{and} \quad \gamma_i^{\Lie}(\Lieg,\Lien) =[{\gamma_{i-1}^{\Lie}(\Lieg,\Lien)},{\Lieg}]_{\Lie}, \quad i \geq 2.
\]
We use the notation $\gamma_i^{\Lie}(\Lieg)$ instead of $\gamma_i^{\Lie}(\Lieg,\Lieg), 1 \leq i \leq n$.
The Leibniz algebra $\Lieg$ is said to be \emph{$\Lie$-nilpotent relative to $\Lien$ of class $c$} if\ $\gamma_{c+1}^{\Lie}(\Lieg, \Lien) = 0$ and $\gamma_c^{\Lie}(\Lieg, \Lien) \neq 0$.

\begin{remark}\label{R:identity}
	Note that from the Leibniz identity we can deduce that~$\big[x, [y, z]_{\Lie}\big] = 0$ is also an identity. This means that it is not interesting at all to study $\Lie$-solvability, since the third ideal will always be trivial.
\end{remark}

\section{\texorpdfstring{$\Lie$-Central derivations and $\Lie$-centroids}{Lie-Central derivations and Lie-centroids}} \label{centroids}

In this section we recall some notions and results from~\cite{BCP} and we provide some new results concerning $\Lie$--central derivations and $\Lie$-centroids.

\begin{definition} A linear map $d \colon \Lieg \rightarrow \Lieg$ of a Leibniz algebra $\Lieg$ is said to be a \emph{$\Lie$-derivation} if for all $x, y \in \Lieg$, the following condition holds:
\[
	d([x,y]_{\Lie})=[d(x),y]_{\Lie} + [x, d(y)]_{\Lie}
\]
\end{definition}

We denote by $\Der^{\Lie}(\Lieg)$ the set of all $\Lie$-derivations of~$\Lieg$ which can be equipped with a structure of Lie algebra by means of the usual bracket
\[
[d_1, d_2] = d_1 \circ d_2 - d_2 \circ d_1, \text{ for all } d_1, d_2 \in \Der^{\Lie}(\Lieg).
\]

\begin{example}\label{derivation1}\
\begin{enumerate}
\item[(a)] Absolute derivations of a Leibniz algebra $\Lieg$, i.e., linear maps $d \colon {\Lieg} \to {\Lieg}$ such that $d([x,y]) = [d(x), y]+[x,d(y)]$, for $x, y \in {\Lieg}$, are also $\Lie$-derivations.

\item[(b)] If $\Lieg$ is a Lie algebra, then every linear map $d \colon {\Lieg} \to {\Lieg}$ is a $\Lie$-derivation.

\item[(c)] Let $\Lieg$ be the two-dimensional Leibniz algebra with basis $\{e, f\}$ and bracket operation $[e, f]=e$ and zero elsewhere (see~\cite{Cu}). The linear maps $d \colon {\Lieg} \to {\Lieg}$ represented by a matrix of the form $\left( \begin{array}{cc} a & 0 \\ 0 & 0 \end{array} \right)$ are $\Lie$-derivations.

 \item[(d)] Let $\Lieg$ be the two-dimensional Leibniz algebra with basis $\{e, f\}$ and bracket operation $[f, f]= \lambda e,~ \lambda \in \K/\K^2$, and zero elsewhere (see~\cite{Cu}). The linear maps $d \colon {\Lieg} \to {\Lieg}$ represented by a matrix of the form $\left( \begin{array}{cc} 2a & b \\ 0 & a \end{array} \right)$ are $\Lie$-derivations, but it is not an absolute derivation.
\end{enumerate}
\end{example}

\begin{definition}[\cite{BCP}] \label{Lie central der}
A $\Lie$-derivation $d \colon \Lieg \to \Lieg$ of a Leibniz algebra $\Lieg$ is said to be a \emph{$\Lie$--central derivation} if its image is contained in the $\Lie$-centre of $\Lieg$.
\end{definition}

We denote the set of all $\Lie$--central derivations of a Leibniz algebra $\Lieg$ by $\Der_z^{\Lie}(\Lieg)$. Obviously $\Der_z^{\Lie}(\Lieg)$ is a subalgebra of $\Der^{\Lie}(\Lieg)$ and every element of $\Der_z^{\Lie}(\Lieg)$ annihilates $\gamma_2^{\Lie}(\Lieg)$. Moreover, ${\Der}_z^{\Lie}({\Lieg}) = C_{{\Der}^{\Lie}({\Lieg})}\big((\RR + \LL)({\Lieg})\big)$, where $\LL({\Lieg})$ is formed by $L_x$, the left multiplication operators $L_x(y) = [x, y]$; and $\RR ({\Lieg})$ by its right counterparts $R_x$.

\begin{example}
The $\Lie$-derivation given in Example~\ref{derivation1}~(c) is not a $\Lie$--central derivation, except in the case $a = 0$. The $\Lie$-derivation given in Example~\ref{derivation1}~(d) is a $\Lie$--central derivation when $a=0$.
\end{example}

\begin{definition}[\cite{BCP}]
The \emph{$\Lie$-centroid} of a Leibniz algebra $\Lieg$ is the set of all linear maps $d \colon \Lieg \to \Lieg$ satisfying the identities
\[
	d([x,y]_{\Lie}) = [d(x), y]_{\Lie} = [x, d(y)]_{\Lie},
\]
for all $x, y \in \Lieg$. We denote this set by $\Gamma^{\Lie}(\Lieg)$.
\end{definition}

\begin{example}\label{derivation}\
\begin{enumerate}
\item[(a)] If $\Lieg$ is a Lie algebra, then every linear map $d \colon {\Lieg} \to {\Lieg}$ is a $\Lie$-centroid.

\item[(b)] Let $\Lieg$ be the two-dimensional Leibniz algebra with basis $\{e, f\}$ and bracket operation $[e, f]=e$ and zero elsewhere (see~\cite{Cu}). The $\Lie$-centroid of $\Lieg$ are the linear maps $d \colon {\Lieg} \to {\Lieg}$ represented by the matrix of the form $\left( \begin{array}{cc} a & 0 \\ 0 & a \end{array} \right)$.

 \item[(c)] Let $\Lieg$ be the two-dimensional Leibniz algebra with basis $\{e, f\}$ and bracket operation $[f, f]= \lambda e,~ \lambda \in \K/\K^2$, and zero elsewhere (see~\cite{Cu}). The linear maps $d \colon {\Lieg} \to {\Lieg}$ represented by a matrix of the form $\left( \begin{array}{cc} a & b \\ 0 & a \end{array} \right)$ are $\Lie$-centroids.
\end{enumerate}
\end{example}

\begin{proposition}[{\cite[Proposition 4.2]{BCP}}] \label{intersection}
For any Leibniz algebra $\Lieg$, $\Gamma^{\Lie}(\Lieg)$ is a subalgebra of $\End_{\K}({\Lieg})$ such that $\Der_z^{\Lie}(\Lieg) = \Der^{\Lie}(\Lieg) \cap \Gamma^{\Lie}(\Lieg)$.
\end{proposition}

\begin{theorem}
	Let $\{ e_k \}$ be a basis of ${\Der}_z^{\Lie}({\Lieg})$ and $\{ \varphi_j \}$ be a maximal subset of~$\Gamma^{\Lie}({\Lieg})$ such that $\{ \varphi_{j \mid \gamma_2^{\Lie}(\Lieg)} \}$ is linearly independent and let $\Psi$ denote the subspace spanned by $\{ \varphi_j \}$. Then $\{ e_k, \varphi_j \}$ is a basis of $\Gamma^{\Lie}({\Lieg})$ and $\Gamma^{\Lie}({\Lieg}) = {\Der}_z^{\Lie}({\Lieg}) \oplus \Psi$.
\end{theorem}
\begin{proof}
Since $\{ \varphi_{j \mid \gamma_2^{\Lie}(\Lieg)} \}$ is linearly independent, then $\{ \varphi_{j} \}$ is linearly independent in $\Gamma^{\Lie}({\Lieg})$. Moreover, by construction and \cite[Proposition 4.2]{BCP}, $\{ e_k, \varphi_j \}$ is a linearly independent set in $\Gamma^{\Lie}({\Lieg})$.
For $\varphi \in \Gamma^{\Lie}({\Lieg})$, since $\{ \varphi_{j \mid \gamma_2^{\Lie}(\Lieg)} \}$ is a basis of the vector space $\{ \varphi_{\mid \gamma_2^{\Lie}(\Lieg)} \mid \varphi \in \Gamma^{\Lie}({\Lieg})\}$, then $\varphi_{\mid \gamma_2^{\Lie}(\Lieg)} = \sum_{j \in J} c_j \varphi_{j \mid \gamma_2^{\Lie}(\Lieg)}$ ($J$ denotes a finite set of indexes and $c_j \in \K$, $j \in J$). Thus, for any $x, y \in {\Lieg}$, we have:
\[
	0 = \bigg( \varphi - \sum_{j \in J} c_j \varphi_{j} \bigg)([x, y]_{\Lie}) = \bigg[ \bigg(\varphi - \sum_{j \in J} c_j \varphi_{j} \bigg)(x), y \bigg]_{\Lie},
\]
i.e., $\left(\varphi - \sum_{j \in J} c_j \varphi_{j} \right)({\Lieg}) \subseteq \ZLie({\Lieg})$. Moreover, it is easy a routine computation that $\varphi - \sum_{j \in J} c_j \varphi_{j} \in {\Der}_z^{\Lie}({\Lieg})$. Consequently, $\left(\varphi - \sum_{j \in J} c_j \varphi_{j} \right) = \sum_{k \in K} b_k e_k$, and therefore $\varphi = \sum_{j \in J} c_j \varphi_{j} + \sum_{k \in K} b_k e_k$.
\end{proof}

\begin{theorem}
Let $\pi \colon \Lieg \twoheadrightarrow \Lieh$ be a surjective homomorphism of Leibniz algebras. For any $f \in \underline{{\End}}(\Lieg) = \{ g \in {\End}(\Lieg) \mid g\big(\Ker(\pi)\big) \subseteq {\Ker}(\pi) \}$ there exists a unique~${\overline{f} \in {\End}(\Lieh)}$ such that $\pi \circ f = \overline{f} \circ \pi$.

Moreover, the following statements hold:
\begin{enumerate}
\item[(i)] The homomorphism $\pi^* \colon \underline{{\End}}(\Lieg) \to {\End}(\Lieh)$, $f \mapsto \overline{f}$, satisfies the following properties:
 \begin{enumerate}
 \item[(a)] $\pi^* \big((\RR+\LL)(\Lieg)\big) = (\RR+\LL)(\Lieh)$.
 \item[(b)] $\pi^*\left(\Gamma^{\Lie}(\Lieg) \cap \underline{{\End}}(\Lieg) \right) \subseteq \Gamma^{\Lie}(\Lieh)$.
 \item[(c)] There is a homomorphism $\pi_{\Gamma^{\Lie}} \colon \Gamma^{\Lie}(\Lieg) \cap \underline{{\End}}(\Lieg) \to \Gamma^{\Lie}(\Lieh)$, $\varphi \mapsto \overline{\varphi}$.
 \item[(d)] If ${\Ker}(\pi) = \ZLie(\Lieg)$, then every $\varphi \in \Gamma^{\Lie}(\Lieg)$ leaves ${\Ker}(\pi)$ invariant, i.e., $\pi_{\Gamma^{\Lie}}$ is defined on all of $\Gamma^{\Lie}(\Lieg)$.
 \end{enumerate}
\item[(ii)] If ${\Ker}(\pi) \subseteq \ZLie(\Lieg)$, for any $\varphi \in \Gamma^{\Lie}(\Lieg) \cap \underline{{\End}}(\Lieg)$ such that $\pi_{\Gamma^{\Lie}}(\varphi) = 0$, then $\varphi\big(\gamma_2^{\Lie}(\Lieg)\big) = 0$.

\end{enumerate}
\end{theorem}
\begin{proof}
The existence of $\overline{f}$ is a consequence of the following commutative diagram:
\[
\xymatrix{
0 \ar[r] & {\Ker}(\pi) \ar[r] \ar@{^(->}[d] & \Lieg \ar[r]^-{\pi} \ar[d]^f & \Lieh \ar[r] \ar@{-->}[d]^-{\overline{f}} & 0\\
0 \ar[r] & {\Ker}(\pi) \ar[r] & \Lieg \ar[r]^-{\pi} & \Lieh \ar[r] & 0
}
\]

{(i) (a)}
Let $x \in \Lieg$ and $y \in \Lieh$ such that $\pi(x) = y$. Then,
\[
	\pi^*\big((\RR+\LL)(x)\big) = \pi^*({\RR_{x}}+{\LL_{x}}) = {\overline{\RR}_{x}} + {\overline{\LL}_{x}} = {\RR_{y}} + {\LL_{y}} \in (\RR+\LL)({y}).
\]

{(b)}
Since we have the isomorphism $\Lieh \cong \Lieg/{\Ker}(\pi)$, for any $f \in \Gamma^{\Lie}(\Lieg) \cap \underline{{\End}}(\Lieg)$ we have that $\overline{f}(y) = f(x)+ {\Ker}(\pi)$ where $y = x + {\Ker}(\pi)$. Then, an easy computation shows that~$\overline{f} \in \Gamma^{\Lie}(\Lieh)$.

{(c)} Direct checking.

{(d)} For any $x \in {\Ker}(\pi) = \ZLie(\Lieg)$, we have that $[\varphi(x),y]_{\Lie} = \varphi([x,y]_{\Lie}) = 0$, for all $y \in \Lieg$, i.e., $\varphi(x) \in \ZLie(\Lieg)$.

{(ii)} Let $\varphi \in \Gamma^{\Lie}(\Lieg) \cap \underline{{\End}}(\Lieg)$ be such that $\pi_{\Gamma^{\Lie}}(\varphi) = 0$. Then, we know that~$\varphi\big(\Lieg\big) \subseteq {\Ker}(\pi) \subseteq \ZLie(\Lieg)$.
For every~$x \in \gamma_2^{\Lie}(\Lieg)$, $x = \sum_i k_i [x'_{i},x''_i]_{\Lie}$, with~$k_i \in \K, x'_{i},x''_i \in \Lieg$, we have
\[
 \varphi(x) = \sum_i k_i \varphi([x'_{i},x''_i]_{\Lie}) = \sum_i k_i [\varphi(x'_{i}),x''_i]_{\Lie} = 0. \qedhere
\]

\end{proof}

\begin{proposition} \label{3.6}
	For a Leibniz algebra $\Lieg$, the elements of its $\Lie$-centroid commute when applied to $\gamma_2^{\Lie}(\Lieg)$.
\end{proposition}
\begin{proof}
Let $x = \sum_i k_i[x_i', x_i''] \in \gamma_2^{\Lie}(\Lieg)$ and let $\varphi$, $\psi \in \Gamma^{\Lie}(\Lieg)$. Then,
\begin{align*}
	\varphi \circ \psi (x) &= \varphi \circ \psi\big( \sum_i k_i[x_i', x_i''] \big) = \sum_i k_i \big(\varphi \circ \psi ([x_i', x_i''])\big) \\
	{} &=  \sum_i k_i \big([\varphi(x_i'), \psi(x_i'')]\big) = \sum_i k_i \big([\varphi(x_i'), \psi(x_i'')]\big) = \psi \circ \varphi (x). \qedhere
\end{align*}
\end{proof}

\begin{proposition}
	Let $\Lieg$ be a  Leibniz algebra. Then:
	\begin{enumerate}
		\item[(i)] $\gamma_2^{\Lie}(\Lieg)$ is indecomposable if and only if the only idempotents of $\Gamma^{\Lie}(\gamma_2^{\Lie}(\Lieg))$ are 0 and~$\id$.
		
		\item[(ii)] For every $\varphi \in \Gamma^{\Lie}(\Lieg)$ and any {\Lie}-invariant $\K$-bilinear form $f$ of $\Lieg$ (i.e.,~$f$ is a $\K$-bilinear form on $\Lieg$ satisfying that $f([a,c]_{\Lie}, b) + f(a, [b,c]_{\Lie})=0$, for all $a, b, c \in \Lieg$) the following equality holds for any $x \in \gamma_2^{\Lie}(\Lieg), b \in \Lieg$:
		\[
		f\big(\varphi(x), b\big) = f\big(x, \varphi(b)\big).
		\]
	\end{enumerate}
\end{proposition}
\begin{proof}
	(i)
	Assume that $\gamma_2^{\Lie}(\Lieg)$ has a decomposition $\gamma_2^{\Lie}(\Lieg) = {\Lieg}_1 \oplus {\Lieg}_2$. We can take the idempotent $\varphi \in \Gamma^{\Lie}({\gamma_2^{\Lie}(\Lieg)})$ such that $\varphi_{\mid {\Lieg}_1} = \id, \varphi_{\mid {\Lieg}_2} = 0$.
	
	Let us assume now that $\varphi \in \Gamma^{\Lie}(\gamma_2^{\Lie}(\Lieg))$ is an idempotent such that $\varphi \neq 0$ and~$\varphi \neq \id$. Then $\varphi^2(y) = \varphi(y)$, for all $y \in \gamma_2^{\Lie}(\Lieg)$. We claim that ${\Ker}(\varphi)$ and~$\Ima(\varphi)$ are two-sided ideals of $\gamma_2^{\Lie}(\Lieg)$. Indeed, for any $x \in {\Ker}(\varphi)$, $y \in \gamma_2^{\Lie}(\Lieg)$,
	\[
	\varphi([x, y]) = \varphi([x, \sum_i k_i [y'_{i}, y''_{i}]_{\Lie}]) = \varphi(0)=0,
	\]
	hence $[x, y] \in {\Ker}(\varphi)$, for all $y \in \gamma_2^{\Lie}(\Lieg)$. On the other hand,
	\begin{align*}
		\varphi([y, x]) &= \varphi([\sum_i k_i [y'_{i}, y''_{i}]_{\Lie},x]) = \varphi([\sum_i k_i [y'_{i}, y''_{i}]_{\Lie},x]_{\Lie}) \\
		{}&= [\sum_i k_i [y'_{i}, y''_{i}]_{\Lie},\varphi(x)]_{\Lie}=0.
	\end{align*}
	
	Let us show that $\Ima(\varphi)$ is a two-sided ideal of $\gamma_2^{\Lie}(\Lieg)$. Indeed, for any~${x,y \in \gamma_2^{\Lie}(\Lieg)}$ we have
	\[
	[\varphi(x),y] = [\varphi(x), \sum_i k_i [y'_{i}, y''_{i}]_{\Lie}] = \sum_i k_i [\varphi(x),[y'_{i}, y''_{i}]_{\Lie}] = 0 \in \Ima(\varphi).
	\]
	Then,
	\begin{align*}
		[y, \varphi(x)] &=[\sum_i k_i [y'_{i}, y''_{i}]_{\Lie}, \varphi(x)] = \sum_i k_i \big[[y'_{i}, y''_{i}]_{\Lie}, \varphi(x)\big]_{\Lie} \\
		{}&= \sum_i k_i \varphi(\big[[y'_{i}, y''_{i}]_{\Lie}, x\big]_{\Lie}) \in \Ima(\varphi).
	\end{align*}
	
	Moreover, we can see that ${\Ker}(\varphi) \cap \Ima(\varphi) = 0$. Indeed, if $x \in {\Ker}(\varphi) \cap \Ima(\varphi)$, then there exists $y \in \gamma_2^{\Lie}(\Lieg)$ such that $x = \varphi(y)$, hence $0 = \varphi(x) = \varphi^2(y) = \varphi(y) = x$.
	In addition $y = x + {\Ker}(\varphi)$, with $x \in \Ima(\varphi)$. Consequently, $ \gamma_2^{\Lie}(\Lieg) = {\Ker}(\varphi) \oplus \Ima(\varphi)$, contradicting the indecomposability of $\gamma_2^{\Lie}(\Lieg)$.
	
	(ii) Let us consider $f$ to be a $\Lie$-invariant $\K$-bilinear form on \Lieg. For any~${\varphi \in \Gamma^{\Lie}(\Lieg)}$, and~$x\in \gamma_2^{\Lie}(\Lieg), a_i, b, c_i  \in {\Lieg}$, we have:
	\begin{align*}
		f(\varphi(x), b) &= f \left( \varphi \left( \sum_i k_i [a_i,c_i]_{\Lie} \right), b \right) = f \left( \sum_i k_i [a_i,\varphi \left( c_i \right)]_{\Lie}, b \right) \\
		{}&= - f \left( \sum_i k_i [a_i,[b, \varphi \left( c_i \right)]_{\Lie}] \right) = - f \left( \sum_i k_i [a_i,\varphi \left( [b, c_i]_{\Lie} \right) ] \right) \\
		{}&= - f \left( \sum_i k_i [a_i,[\varphi \left( b \right), c_i]_{\Lie}] \right) = f \left( \sum_i k_i [a_i, c_i]_{\Lie}, \varphi(b) \right) = f\big(x, \varphi(b)\big).
	\end{align*}
\end{proof}


\section{\texorpdfstring{Generalised $\Lie$-derivations and quasi-$\Lie$-centroids}{Generalised Lie-derivations and quasi-Lie-centroids}} \label{generalised}

In this section, we introduce the notions of generalised $\Lie$-derivations and quasi-$\Lie$-centroids and analyse their interplay with $\Lie$-centroids.

\begin{definition}
Let $\Lieg$ be a Leibniz algebra. A linear map $f \colon {\Lieg} \to {\Lieg}$ is called a \emph{generalised $\Lie$-derivation of $\Lieg$} if there exist linear maps $f', f'' \colon {\Lieg} \to {\Lieg}$ such that
\[
	[f(x),y]_{\Lie} + [x, f''(y)]_{\Lie} = f'([x,y]_{\Lie}),
\]
for all $x, y \in {\Lieg}$. We denote by $\GenDer^{\Lie}({\Lieg})$ the set of all generalised $\Lie$-derivations of $\Lieg$.

We say that $f$ is a \emph{quasi-$\Lie$-derivation of $\Lieg$} if there exists a linear map $f' \colon {\Lieg} \to {\Lieg}$ such that
\[
	[f(x),y]_{\Lie} + [x, f(y)]_{\Lie} = f'([x,y]_{\Lie}),
\]
for all $x, y \in {\Lieg}$. We denote by $\QDer^{\Lie}({\Lieg})$ the set of all quasi-$\Lie$-derivations of $\Lieg$.
\end{definition}

It is easy to check that $\GenDer^{\Lie}({\Lieg})$ and $\QDer^{\Lie}({\Lieg})$ are subalgebras of $\End_{\K}({\Lieg})$. In fact, we have the following inclusion tower:
\[
	\Der_z^{\Lie}({\Lieg}) \subseteq \Der^{\Lie}({\Lieg}) \subseteq \QDer^{\Lie}({\Lieg}) \subseteq \GenDer^{\Lie}({\Lieg}) \subseteq \End_{\K}({\Lieg}).
\]

\begin{example} Let $\Lieg$ be the two-dimensional Leibniz algebra in Example~\ref{derivation1}~(c).
\begin{enumerate}
\item[(a)] Let $f, f' \colon {\Lieg} \to {\Lieg}$ be the linear maps represented by the matrices $\left( \begin{array}{cc} 0 & 0 \\ 0 & 1 \end{array} \right)$ and $\left( \begin{array}{cc} 1 & 1 \\ 0 & 1 \end{array} \right)$, respectively. It can be checked that $f \in \QDer^{\Lie}({\Lieg})$, although~${f \notin \Der^{\Lie}({\Lieg})}$.

\item [(b)] Consider the linear maps $f, f', f'' \colon {\Lieg} \to {\Lieg}$ represented by the matrices $\left( \begin{array}{cc} 1 & 1 \\ 0 & 0 \end{array} \right)$, $\left( \begin{array}{cc} 0 & -1 \\ 0 & -1 \end{array} \right)$ and $\left( \begin{array}{cc} 0 & 1 \\ 0 & 0 \end{array} \right)$, respectively. It is easy to check that $f \in \GenDer^{\Lie}({\Lieg})$, but $f \notin \QDer^{\Lie}({\Lieg})$.
\end{enumerate}
\end{example}

\begin{definition}
The \emph{quasi-$\Lie$-centroid} of a Leibniz algebra $\Lieg$ is the set of all linear maps $d \colon {\Lieg} \to {\Lieg}$ satisfying
\[
	[d(x), y]_{\Lie} = [x, d(y)]_{\Lie}
\]
for all $x, y \in {\Lieg}$. We denote this set by $\QQ{\Gamma}^{\Lie}({\Lieg})$.
\end{definition}

\begin{example}
Let $\Lieg$ be the two-dimensional Leibniz algebra in Example~\ref{derivation1}~(c). The linear maps $f \colon {\Lieg} \to {\Lieg}$ represented by the matrices $\left( \begin{array}{cc} a & b \\ c & a \end{array} \right), a, b, c \in \mathbb{C}$, are quasi-$\Lie$-centroids.
\end{example}

\begin{lemma} \label{lemma 1}
Let $\Lieg$ be a Leibniz algebra. We have the following inclusions,
\begin{enumerate}
\item[(i)] $[\Der^{\Lie}({\Lieg}), {\Gamma}^{\Lie}({\Lieg})] \subseteq {\Gamma}^{\Lie}({\Lieg})$,
\item[(ii)] $[\QQ{\Gamma}^{\Lie}({\Lieg}), \QDer^{\Lie}({\Lieg})] \subseteq \QQ{\Gamma}^{\Lie}({\Lieg})$,

\noindent $[ \QDer^{\Lie}({\Lieg}), \QQ{\Gamma}^{\Lie}({\Lieg})] \subseteq \QQ{\Gamma}^{\Lie}({\Lieg})$,
\item[(iii)] $[\QQ{\Gamma}^{\Lie}({\Lieg}), \QQ{\Gamma}^{\Lie}({\Lieg})] \subseteq \QDer^{\Lie}({\Lieg})$,
\item[(iv)] ${\Gamma}^{\Lie}({\Lieg}) \subseteq \QDer^{\Lie}({\Lieg})$,
\item[(v)] $\QDer^{\Lie}({\Lieg}) + \QQ{\Gamma}^{\Lie}({\Lieg}) \subseteq \GenDer^{\Lie}({\Lieg})$.
\end{enumerate}
\end{lemma}

\begin{proof}
(i) Let $d \in \Der^{\Lie}({\Lieg}), d' \in {\Gamma}^{\Lie}({\Lieg})$, then for all $x, y \in {\Lieg}$:
\begin{align*}
	[d,d']([x,y]_{\Lie}) =& (dd' -d'd)([x,y]_{\Lie}) = d[d'(x),y]_{\Lie} - d'([d(x),y]_{\Lie} + [x, d(y)]_{\Lie})\\
	{} = & [dd'(x),y]_{\Lie} - [d'd(x), y]_{\Lie} = \big[[d, d'](x), y\big]_{\Lie}.
\end{align*}
A similar computation shows the other equality.

(ii) Let $d \in \QQ{\Gamma}^{\Lie}({\Lieg}), d' \in \QDer^{\Lie}({\Lieg})$, then for all $x, y \in {\Lieg}$:
\begin{align*}
	\big[[d,d'](x),y\big]_{\Lie} &= [(dd' -d'd)(x),y]_{\Lie} = [dd'(x),y]_{\Lie} - [d'd(x),y]_{\Lie} \\
	{}&= [d'(x), d(y)]_{\Lie} - d_1[d(x),y]_{\Lie} + [d(x), d'(y)]_{\Lie},\\
	[x,[d,d'](y)]_{\Lie} &= [x,(dd' -d'd)(y)]_{\Lie} = [x, dd'(y)]_{\Lie} - [x, d'd(y)]_{\Lie} \\
	{}&= [d(x), d'(y)]_{\Lie} - d_1[x, d(y)]_{\Lie} +[d'(x), d(y)]_{\Lie}.
\end{align*}
The second inclusion can be checked analogously.

(iii) Let $d, d' \in \QQ{\Gamma}^{\Lie}({\Lieg})$. Since

\begin{align*}
	\big[[d,d'](x),y\big]_{\Lie} + [x, [d,d'](y)]_{\Lie} &= [d'(x),d(y)]_{\Lie} - [d(x),d'(y)]_{\Lie} \\
&\quad + [d(x),d'(y)]_{\Lie} - [d'(x),d(y)]_{\Lie}=0,
\end{align*}
then $[d,d'] \in \QDer^{\Lie}({\Lieg})$.

(iv) Let $d \in {\Gamma}^{\Lie}({\Lieg})$, then $[d(x),y]_{\Lie}+[x,d(y)]_{\Lie} = 2 d([x,y]_{\Lie})$, and therefore~${d \in \QDer^{\Lie}({\Lieg})}$, where $d' = 2d$.

(v) Let $d \in \QDer^{\Lie}({\Lieg}), d' \in \QQ{\Gamma}^{\Lie}({\Lieg})$. Since
\begin{align*}
	[(d+d')(x), y]_{\Lie} &= d_1([x,y]_{\Lie}) - [x, d(y)]_{\Lie} + [x, d'(y)]_{\Lie} \\
	{}&= d_1([x,y]_{\Lie}) - [x, (d-d')(y)]_{\Lie},
\end{align*}
then $[(d+d')(x), y]_{\Lie} + [x, (d-d')(y)]_{\Lie} = d_1([x,y]_{\Lie})$, i.e., $d+d' \in \GenDer^{\Lie}({\Lieg})$.
\end{proof}

\begin{theorem}
Let $\Lieg$ be a Leibniz algebra. Then
\[
[{\Gamma}^{\Lie}({\Lieg}), \QQ{\Gamma}^{\Lie}({\Lieg})] \subseteq \End_{\K}\big({\Lieg}, \ZLie({\Lieg})\big).
\]
Moreover, if $ \ZLie({\Lieg}) = 0$, then $[{\Gamma}^{\Lie}({\Lieg}), \QQ{\Gamma}^{\Lie}({\Lieg})]=0$.
\end{theorem}

\begin{proof}
Let $d \in {\Gamma}^{\Lie}({\Lieg}), d' \in \QQ{\Gamma}^{\Lie}({\Lieg})$. Then,
\begin{align*}
\big[[d,d'](x), y\big]_{\Lie}&=[dd'(x),y]_{\Lie} - [d'd(x), y]_{\Lie} \\
{}&= d([d'(x),y]_{\Lie}) - [d(x), d'(y)]_{\Lie} \\
{}&= d([x, d'(y)]_{\Lie}) - [d(x), d'(y)]_{\Lie} = 0,
\end{align*}
i.e., $[d,d'](x) \in \ZLie({\Lieg})$.
\end{proof}

\begin{theorem}
Let $\Lieg$ be a Leibniz algebra. Then ${\Gamma}^{\Lie}({\Lieg}) = \QDer^{\Lie}({\Lieg}) \cap \QQ{\Gamma}^{\Lie}({\Lieg})$.
\end{theorem}
\begin{proof}
Let $f \in \QDer^{\Lie}({\Lieg}) \cap \QQ{\Gamma}^{\Lie}({\Lieg})$, then there exists a linear map $f' \colon {\Lieg} \to {\Lieg}$ such that $f'([x,y]_{\Lie}) = [f(x),y]_{\Lie} + [x, f(y)]_{\Lie}$.

We claim that $f \in {\Gamma}^{\Lie}({\Lieg})$. Indeed, $[f(x),y]_{\Lie} = [x, f(y)]_{\Lie}$ since $f \in \QQ{\Gamma}^{\Lie}({\Lieg})$. On the other hand, $f([x,y]_{\Lie}) = [f(x), y]_{\Lie}$ since $f'([x,y]_{\Lie}) = 2 [f(x), y]_{\Lie}$, therefore $f' = 2f$ provides the needed equality.

Conversely, $f \in {\Gamma}^{\Lie}({\Lieg})$ implies that $f([x,y]_{\Lie}) = [f(x), y]_{\Lie}=[x, f(y)]_{\Lie}$, hence~$f \in \QQ{\Gamma}^{\Lie}({\Lieg})$. Moreover $[f(x), y]_{\Lie}+[x, f(y)]_{\Lie} = 2 f([x,y]_{\Lie})$, so just taking~$f' = 2f$ we conclude the proof.
\end{proof}

\begin{lemma} \label{lemma 2}
Let $\Liem$ and $\Lien$ be two two-sided ideals of a Leibniz algebra $\Lieg$, such that~${{\Lieg} = {\Liem} \oplus {\Lien}}$. Then:
\begin{enumerate}
\item[(i)] $\ZLie({\Lieg}) = \ZLie({\Liem}) \oplus \ZLie({\Lien})$,
\item[(ii)] If $\ZLie({\Lieg})=0$ then,
\begin{itemize}
\item[(a)] $\Der^{\Lie}({\Lieg}) = \Der^{\Lie}({\Liem}) \oplus \Der^{\Lie}({\Lien})$.
\item[(b)] $\GenDer^{\Lie}({\Lieg}) = \GenDer^{\Lie}({\Liem}) \oplus \GenDer^{\Lie}({\Lien})$.
\item[(c)] $\QDer^{\Lie}({\Lieg}) = \QDer^{\Lie}({\Liem}) \oplus \QDer^{\Lie}({\Lien})$.
\item[(d)] ${\Gamma}^{\Lie}({\Lieg}) = {\Gamma}^{\Lie}({\Liem}) \oplus {\Gamma}^{\Lie}({\Lien})$.
\item[(e)] ${\QQ\Gamma}^{\Lie}({\Lieg}) = \QQ{\Gamma}^{\Lie}({\Liem}) \oplus \QQ{\Gamma}^{\Lie}({\Lien})$.
\end{itemize}

\end{enumerate}
\end{lemma}
\begin{proof}
(i) Let $m \in \ZLie({\Liem}), n \in \ZLie({\Lien})$, then for any $y = m'+n' \in {\Lieg}$, we have:
\[
[m+n,y]_{\Lie} = [m,m']_{\Lie}+ [m,n']_{\Lie}+ [n,m']_{\Lie}+ [n,n']_{\Lie}=0,
\]
i.e., $m+n \in \ZLie({\Lieg})$.

Conversely, for any $x = m+n \in \ZLie({\Lieg})$ and for all $y = m'+n' \in {\Lieg}$, we have~$0 = [x,y]_{\Lie} = [m,m']_{\Lie}+[n,n']_{\Lie}$, hence $[m,m']_{\Lie}=0, [n,n']_{\Lie}=0$ for all~$m' \in {\Liem}$, $n' \in {\Lien}$, i.e., $m \in \ZLie({\Liem})$, $n \in \ZLie({\Lien})$.

(ii) (a) Let $d \in \Der^{\Lie}({\Lieg})$. For any $m \in {\Liem}, n \in {\Lien}$, we have
\[
[d(m),n]_{\Lie} = d([m,n]_{\Lie}) - [m,d(n)]_{\Lie} = -[m, m_1]_{\Lie},
\]
for some $m_1 \in {\Liem}$.
Assume that $d(m) = m'+n'$, then
\[
-[m,m_1]_{\Lie} = [d(m),n]_{\Lie} = [m'+n',n]_{\Lie}= [n',n]_{\Lie},
\]
hence $[n',n]_{\Lie}=0$ for all $n \in {\Lien}$, so $n' \in \ZLie({\Lien})=0$. Consequently $d(m) = m'$, i.e.,~$d({\Liem}) \subseteq {\Liem}$.
Similarly, $d({\Lien}) \subseteq {\Lien}$.

Conversely, let $d_1 \in \Der^{\Lie}({\Liem})$, $d_2 \in \Der^{\Lie}({\Lien})$, then the derivation $d \colon {\Lieg} \to {\Lieg}$, defined by $d(g) = d_1(m) + d_2(n)$, for all $g = m+n \in {\Lieg}$, is a $\Lie$-derivation of $\Lieg$.

The other statements are obtained in a similar way.
\end{proof}

\begin{proposition}
Let $\Lieg$ be a Leibniz algebra. Then, ${\QQ\Gamma}^{\Lie}({\Lieg}) + [{\QQ\Gamma}^{\Lie}({\Lieg}), {\QQ\Gamma}^{\Lie}({\Lieg})]$ is a subalgebra of $\GenDer^{\Lie}({\Lieg})$.
\end{proposition}
\begin{proof}
By Lemma~\ref{lemma 1}~(iii) and (v), we have:
\[
{\QQ\Gamma}^{\Lie}({\Lieg}) + [{\QQ\Gamma}^{\Lie}({\Lieg}), {\QQ\Gamma}^{\Lie}({\Lieg})] \subseteq {\QQ\Gamma}^{\Lie}({\Lieg}) + \QDer^{\Lie}({\Lieg}) \subseteq \GenDer^{\Lie}({\Lieg}).
\]

Now we show the bracket is closed. Indeed, keeping in mind Lemma~\ref{lemma 1}~(ii) and (iii), we have:
\begin{align*}
\big[{\QQ\Gamma}^{\Lie}({\Lieg}) + [{\QQ\Gamma}^{\Lie}({\Lieg}), {\QQ\Gamma}^{\Lie}({\Lieg})], {\QQ\Gamma}^{\Lie}({\Lieg}) + [{\QQ\Gamma}^{\Lie}({\Lieg}), {\QQ\Gamma}^{\Lie}({\Lieg})]\big] &\subseteq \\
\big[{\QQ\Gamma}^{\Lie}({\Lieg}) + \QDer^{\Lie}({\Lieg}), {\QQ\Gamma}^{\Lie}({\Lieg}) + [{\QQ\Gamma}^{\Lie}({\Lieg}), {\QQ\Gamma}^{\Lie}({\Lieg})]\big] &\subseteq \\
[{\QQ\Gamma}^{\Lie}({\Lieg}), {\QQ\Gamma}^{\Lie}({\Lieg})] + \big[{\QQ\Gamma}^{\Lie}({\Lieg}),[{\QQ\Gamma}^{\Lie}({\Lieg}),{\QQ\Gamma}^{\Lie}({\Lieg})]\big]+& \\
[\QDer^{\Lie}({\Lieg}), {\QQ\Gamma}^{\Lie}({\Lieg})] + \big[\QDer^{\Lie}({\Lieg}),[\QDer^{\Lie}({\Lieg}),\QDer^{\Lie}({\Lieg})]\big] &\subseteq \\
{\QQ\Gamma}^{\Lie}({\Lieg}) + [{\QQ\Gamma}^{\Lie}({\Lieg}), {\QQ\Gamma}^{\Lie}({\Lieg})].
\end{align*}
\end{proof}


\section{\texorpdfstring{Almost inner {$\Lie$}-derivations}{Almost inner Lie-derivations}}\label{almost}

In this section we study the context under which almost inner $\Lie$-derivations and central derivations coincide. It is interesting to note that the relative setting requires an additional condition with respect to the absolute context.

\begin{definition}
An \emph{almost inner $\Lie$-derivation} of a Leibniz algebra $\Lieg$ is a $\Lie$-derivation $d \colon {\Lieg} \to {\Lieg}$ such that $d(x) \in [x, \Lieg]_{\Lie}$, for all $x \in {\Lieg}$. We denote by ${\Der}_c^{\Lie}(\Lieg)$ the subspace of ${\Der}^{\Lie}(\Lieg)$ consisting in all almost inner $\Lie$-derivations of $\Lieg$, that is
\[
	{\Der}_c^{\Lie}(\Lieg) = \{ d \in \Der^{\Lie}(\Lieg) \mid d(x) \in [x, \Lieg]_{\Lie}, \text{ for all } x \in \Lieg \}.
\]
Let us denote by $T_c \left( \frac{\Lieg}{\ZLie(\Lieg)}, \gamma_2^{\Lie}(\Lieg) \right)$ the following vector space:
\[
\left \{ f \in \Hom_{\K} \left( \frac{\Lieg}{\ZLie(\Lieg)}, \gamma_2^{\Lie}(\Lieg) \right) \mid f\big(x+ \ZLie(\Lieg)\big) \in [x, {\Lieg}]_{\Lie}, ~\text{for all} ~x \in {\Lieg} \right \}.
\]
\end{definition}

\begin{proposition}
Let $\Lieg$ be a $\Lie$-nilpotent Leibniz algebra with class of $\Lie$-nilpo\-tency~$2$. Then ${\Der}_c^{\Lie}(\Lieg) \cong T_c \left( \frac{\Lieg}{\ZLie(\Lieg)}, \gamma_2^{\Lie}(\Lieg) \right).$
\end{proposition}
\begin{proof}
Let $d \in {\Der}_c^{\Lie}(\Lieg)$, the map $\psi_d \colon \frac{\Lieg}{\ZLie(\Lieg)} \to \gamma_2^{\Lie}(\Lieg)$, $\psi_d\big(x + \ZLie(\Lieg)\big) = d(x)$, is a linear map in $T_c \big( \frac{\Lieg}{\ZLie(\Lieg)}, \gamma_2^{\Lie}(\Lieg) \big)$.

Since $\gamma_3^{\Lie}(\Lieg) =0$, then ${\Der}_c^{\Lie}(\Lieg)$ is an abelian Leibniz algebra and the vector space $T_c \big( \frac{\Lieg}{\ZLie(\Lieg)}, \gamma_2^{\Lie}(\Lieg) \big)$ is also an abelian Leibniz algebra.

The map $\Psi \colon {\Der}_c^{\Lie}(\Lieg) \to T_c \left( \frac{\Lieg}{\ZLie(\Lieg)}, \gamma_2^{\Lie}(\Lieg) \right), \Psi(d) = \psi_d$, is an isomorphism of abelian Leibniz algebras.
\end{proof}

The reason we cannot study inner $\Lie$-derivations directly is because the {$\Lie$-adjoint} map is not a $\Lie$-derivation in general. Let $\Lieg$ be the free Leibniz algebra generated by $x$, $y$ and $z$. The map $d_x = [x, -]_{\Lie} \colon \Lieg \to \Lieg$ would be a $\Lie$-derivation if and only if the following identity was true:
\[
\big[x, [y, z]_{\Lie}\big]_{\Lie} = \big[[x, y]_{\Lie}, z\big]_{\Lie} + \big[y, [x, z]_{\Lie}\big]_{\Lie},
\]
which expanded together with Remark~\ref{R:identity} would mean that
\[
	\big[[y, z], x\big] + \big[[x, y], z\big] = \big[[y, x], z\big] + \big[[x, z], y\big] + \big[[z, x], y\big].
\]
Adding it to the same identity but interchanging $x$ and $y$, we obtain:
\[
	2\big[[x, y], z\big] + 2\big[[y, x], z\big] = 0.
\]
Therefore, it only makes sense to study inner $\Lie$-derivations in Leibniz algebras $\Lieg$ satisfying
\[
\big[[x, y]_{\Lie}, z\big] = 0
\]
for all $x, y, z \in \Lieg$, which is the same as saying that $\gamma_2^{\Lie}(\Lieg) \subseteq Z(\Lieg)$.
Note that by the symmetry of the $\Lie$-bracket it does not matter to consider left or right adjoint maps. Let $\Lieg$ be a Leibniz algebra satisfying $\gamma_2^{\Lie}(\Lieg) \subseteq Z(\Lieg)$. By $\IDer^{\Lie} = \{d_x \colon x \in {\Lieg} \}$ we denote the set of all inner $\Lie$-derivations of $\Lieg$.

\begin{example} \label{ex 1}
The three-dimensional Leibniz algebra with basis $\{a_1, a_2, a_3 \}$ and bracket operation $[a_3, a_3] = a_1$ and zero elsewhere (class 2 (b) in~\cite{CILL}) is a Leibniz algebra satisfying $\gamma_2^{\Lie}(\Lieg) \subseteq Z(\Lieg)$.
\end{example}

Now we are going to explore the similarities between ${\Der}_c^{\Lie}(\Lieg)$ and ${\Der}_z^{\Lie}(\Lieg)$.

\begin{theorem} \label{equal}
Let $\Lieg$ be a finite-dimensional non-abelian Leibniz algebra satisfying~$\gamma_2^{\Lie}(\Lieg) \subseteq Z(\Lieg)$. Then ${\Der}_c^{\Lie}(\Lieg) = {\Der}_z^{\Lie}(\Lieg)$ if and only if $\ZLie(\Lieg) = \gamma_2^{\Lie}(\Lieg)$ and~${\Der}_c^{\Lie}(\Lieg) \cong T \left( \frac{\Lieg}{\ZLie(\Lieg)}, \gamma_2^{\Lie}(\Lieg) \right)$.
\end{theorem}
\begin{proof}
Assume that ${\Der}_c^{\Lie}(\Lieg) = {\Der}_z^{\Lie}(\Lieg)$.
For any $d \in {\Der}_c^{\Lie}(\Lieg)$, the homomorphism
\[
\psi_d \colon \frac{\Lieg}{\ZLie(\Lieg)} \to \gamma_2^{\Lie}(\Lieg), \qquad \psi_d\big(x + \ZLie(\Lieg)\big) = d(x),
\]
is a linear map. Now we define
\[
\Psi \colon {\Der}_c^{\Lie}(\Lieg) \to T \left( \frac{\Lieg}{\ZLie(\Lieg)}, \gamma_2^{\Lie}(\Lieg) \right), \qquad \Psi(d) = \psi_d.
\]
Clearly $\Psi$ is an isomorphism of Leibniz algebras.

Now, by the isomorphism of $\K$-vector spaces given in \cite[Lemma 3.5]{BCP}, we have
\begin{align*}
	\dim \left( T \left( \frac{\Lieg}{\ZLie(\Lieg)}, \gamma_2^{\Lie}(\Lieg) \right) \right) &= \dim \left( {\Der}_c^{\Lie}(\Lieg) \right) = \dim \left( {\Der}_z^{\Lie}(\Lieg) \right) \\
	{}&= \dim \left( T \left( \frac{\Lieg}{\gamma_2^{\Lie}(\Lieg)}, \ZLie(\Lieg) \right) \right)
\end{align*}
Therefore, since $\gamma_2^{\Lie}(\Lieg) \subseteq \ZLie(\Lieg)$, by a linear algebra argument concerning dimensions we know that $\gamma_2^{\Lie}(\Lieg) = \ZLie(\Lieg)$.

Conversely, assume that $\gamma_2^{\Lie}(\Lieg) = \ZLie(\Lieg)$ and ${\Der}_c^{\Lie}(\Lieg) \cong T \left( \frac{g}{\ZLie(\Lieg)}, \gamma_2^{\Lie}(\Lieg) \right)$.
It is clear that ${\Der}_c^{\Lie}(\Lieg) \subseteq {\Der}_z^{\Lie}(\Lieg)$.
On the other hand, due to~\cite[Lemma 3.5]{BCP}, we have
\begin{align*}
\dim \left( {\Der}_z^{\Lie}(\Lieg) \right) =& \dim \left( T \left( \frac{\Lieg}{\gamma_2^{\Lie}(\Lieg)}, \ZLie(\Lieg) \right) \right) = \dim \left( T \left( \frac{\Lieg}{\ZLie(\Lieg)}, \gamma_2^{\Lie}(\Lieg) \right) \right) \\
{}=& \dim \left( {\Der}_c^{\Lie}(\Lieg) \right)
\end{align*}
which completes the proof.
\end{proof}

\begin{corollary} \label{5.10}
Let $\Lieg$ be a finite-dimensional non-abelian Leibniz algebra satisfying~$\gamma_2^{\Lie}(\Lieg) \subseteq Z(\Lieg)$ such that $\dim \left( \ZLie(\Lieg) \right) = 1$. Then ${\Der}_c^{\Lie}(\Lieg) = {\Der}_z^{\Lie}(\Lieg)$ if and only if $\ZLie(\Lieg) = \gamma_2^{\Lie}(\Lieg)$.
\end{corollary}
\begin{proof}
If $\ZLie(\Lieg) = \gamma_2^{\Lie}(\Lieg)$, then $\IDer^{\Lie}(\Lieg)$ and ${\Der}_c^{\Lie}(\Lieg)$ are subalgebras of ${\Der}_z^{\Lie}(\Lieg)$. Since $\frac{\Lieg}{\ZLie(\Lieg)} \cong \IDer^{\Lie}(\Lieg)$ and by~\cite[Lemma 3.5]{BCP}, we have
\[
	\dim \left( {\Der}_z^{\Lie}(\Lieg) \right) = \dim \left( T \left( \frac{g}{\gamma_2^{\Lie}(\Lieg)}, \ZLie(\Lieg) \right) \right),
\]
hence
\begin{align*}
	\dim \left( {\Der}_z^{\Lie}(\Lieg) \right) &= \dim \left( \frac{\Lieg}{\gamma_2^{\Lie}(\Lieg)} \right) \cdot \dim \left( \ZLie(\Lieg) \right) \\ &= \dim \left( \frac{\Lieg}{\ZLie(\Lieg)} \right) = \dim \left( \IDer^{\Lie}(\Lieg) \right).
\end{align*}
Consequently, ${\Der}_c^{\Lie}(\Lieg) \subseteq {\Der}_z^{\Lie}(\Lieg) = \IDer^{\Lie}(\Lieg) \subseteq {\Der}_c^{\Lie}(\Lieg)$.

The converse is just an immediate consequence of Theorem~\ref{equal}.
\end{proof}

\begin{example} \label{5.7}
The three-dimensional Leibniz algebra with basis $\{ a_1, a_2, a_3 \}$ and bracket operation given by $[a_2,a_2]=[a_3,a_3]=a_1$ and zero elsewhere (class 2 (c) in~\cite{CILL}), is a Leibniz algebra satisfying the requirements of Corollary~\ref{5.10}.
\end{example}

\section{\texorpdfstring{$\Lie$-centroid of a tensor product}{Lie-centroid of a tensor product}}\label{tensor}

In this section we analyse some properties concerning the $\Lie$-centroid of the Leibniz algebra tensor product of a commutative associative algebra and a Leibniz algebra.

Let $A$ be a commutative associative algebra over $\K$. The \emph{centroid of $A$} is the associative subalgebra of ${\End_{\K}}(A)$
\[
	\Gamma(A) = \{ f \in {\End_{\K}}(A) \mid f(ab) = f(a)b = af(b) \}.
\]

For a Leibniz algebra $\Lieg$, let $A \otimes {\Lieg}$ be the tensor product over $\K$ of the underlying vector spaces of $A$ and {\Lieg}. Then $A \otimes {\Lieg}$ can be endowed with a structure of Leibniz algebra with the operation
\[
	[a_1 \otimes g_1, a_2 \otimes g_2] = (a_1 a_2) \otimes [g_1, g_2]
\]
for all $a_1, a_2 \in A$, $g_1, g_2 \in {\Lieg}$. This Leibniz algebra is called the \emph{tensor product Leibniz algebra of $A$ and $\Lieg$}.

Given $f \in {\End}_{\K}(A)$, $\varphi \in {\End}_{\K}(\Lieg)$, the map $f \widetilde{\otimes} \varphi \colon A \otimes {\Lieg} \to A \otimes {\Lieg}$ defined by~$f \widetilde{\otimes} \varphi(a \otimes g) = f(a) \otimes \varphi(g)$ is an endomorphism of the Leibniz algebras.

\begin{corollary}
Let $\Lieg$ be a finite-dimensional non-abelian Leibniz algebra satisfying~$\gamma_2^{\Lie}(\Lieg) \subseteq Z(\Lieg)$ such that $\dim \left( \ZLie(\Lieg) \right) = 1$. Then ${\Der}_c^{\Lie}(\Lieg) = {\Der}^{\Lie}(\Lieg) \cap \Gamma^{\Lie}({\Lieg})$ if and only if $\ZLie(\Lieg) = \gamma_2^{\Lie}(\Lieg)$.
\end{corollary}
\begin{proof}
By Corollary~\ref{5.10}, we know that ${\Der}_c^{\Lie}(\Lieg) ={\Der}_z^{\Lie}(\Lieg)$. Then~\cite[Proposition~4.2]{BCP} concludes the proof.
\end{proof}

\begin{lemma} \label{4.1}
We have the following inclusion:
\[
	\Gamma(A) \otimes \Gamma^{\Lie}(\Lieg) \subseteq \Gamma^{\Lie}(A \otimes \Lieg).
\]
\end{lemma}
\begin{proof} For any $f \in \Gamma(A)$ and $\varphi \in \Gamma^{\Lie}(\Lieg)$ we have:
\begin{align*}
f \widetilde{\otimes} \varphi([a_1 \otimes g_1, a_2 \otimes g_2]_{\Lie}) &= f \widetilde{\otimes} \varphi \big((a_1a_2) \otimes [g_1, g_2] + (a_2 a_1) \otimes [g_2, g_1]\big)\\
{}&= f(a_1 a_2) \otimes \varphi([g_1,g_2]_{\Lie}) = (f(a_1) a_2) \otimes [\varphi(g_1),g_2]_{\Lie}\\
{}&= [f(a_1) \otimes \varphi(g_1), a_2 \otimes g_2]_{\Lie} \\
{}&= [f \widetilde{\otimes} \varphi(a_1 \otimes g_1), a_2 \otimes g_2]_{\Lie}.
\end{align*}

The other equality can be checked in a similar way.
\end{proof}

The following example shows that the converse inclusion does not hold in general.

\begin{example}
Let $A$ be the two-dimensional commutative associative algebra with basis $\{e_1, e_2 \}$ and product $e_1.e_1 = e_1$, $e_1.e_2 = e_2.e_1 = e_2$, $e_2.e_2=0$ (see class~${As}_2^4$ in~\cite{RRB}) and let $\Lieg$ be the two-dimensional Leibniz algebra with basis $\{a_1, a_2 \}$ and bracket operation given by $[a_1, a_2]=a_1$ and zero elsewhere~\cite{Cu}. The $\Lie$-centroid of~$A \otimes {\Lieg}$ are the linear maps represented by a matrix of the form $\left( \begin{array}{cccc} \lambda & 0 & 0 & 0\\ 0 & \lambda & 0 & 0\\ \mu& 0 & \lambda & 0 \\ 0 & \mu & 0 & \lambda \end{array} \right)$.

The element of~$(A \otimes {\Lieg})$ represented by the matrix $\left( \begin{array}{cccc} 1 & 0 & 0 & 0\\ 0 & 1 & 0 & 0\\ 1& 0 & 1 & 0 \\ 0 & 1 & 0 & 1 \end{array} \right)$ cannot be obtained as the tensor product of an element $\Gamma(A)$, whose associated matrix has the form $\left( \begin{array}{cc} \lambda & 0 \\ \mu& \lambda \end{array} \right)$, and an element of $\Gamma^{\Lie}({\Lieg})$, whose associated matrix is of the form $\left( \begin{array}{cc} 0 & 0 \\ 0& \delta \end{array} \right)$.
\end{example}

If $A$ is unital, then the isomorphism $\sigma \colon \Gamma(A) \to A$ given by $\sigma(f) = f(1)$, implies that $\Gamma(A) \cong A$

From now on we assume that $\Lieg$ is a Leibniz algebra over an algebraically closed field $\K$ of characteristic zero and $A$ is a unital commutative associative algebra over~$\K$.

Let $\Psi$ denote the subspace of $\Gamma^{\Lie}(\Lieg)$ spanned by $\{ \varphi_j \}_{j \in J}$, a maximal subset of~$\Gamma^{\Lie}(\Lieg)$ such that $\{ \varphi_{j \mid \gamma_2^{\Lie}(\Lieg)} \}_{j \in J}$ is linearly independent.

\begin{proposition}
We have the following inclusion:
\[
	A \otimes \Psi + {\End}_{\K}(A) \otimes {\Der}_z^{\Lie}(\Lieg) \subseteq \Gamma^{\Lie}(A \otimes \Lieg)
\]
\end{proposition}
\begin{proof}
We just need to show that ${\End}_{\K}(A) \otimes {\Der}_z^{\Lie}(\Lieg) \subseteq \Gamma^{\Lie}(A \otimes \Lieg)$, since the other part is implied by Lemma~\ref{4.1}.
For any $f \in {\End}_{\K}(A)$, and $\varphi \in {\Der}_z^{\Lie}(\Lieg)$, we have:
\begin{align*}
f \widetilde{\otimes} \varphi ([a_1 \otimes g_1, a_2 \otimes g_2]_{\Lie}) &= f(a_1 a_2) \otimes \varphi([g_1, g_2]_{\Lie})\\
&= f(a_1 a_2) \otimes ([\varphi(g_1), g_2]_{\Lie}+ [g_1, \varphi(g_2)]_{\Lie})= 0,\\
[f \widetilde{\otimes} \varphi (a_1 \otimes g_1), a_2 \otimes g_2]_{\Lie} &= f(a_1) a_2 \otimes [\varphi(g_1), g_2]_{\Lie} = 0,\\
[ a_1 \otimes g_1, f \widetilde{\otimes} \varphi (a_2 \otimes g_2)]_{\Lie} &= a_1 f(a_2) \otimes [g_1, \varphi(g_2)]_{\Lie} = 0.
\end{align*}
\end{proof}

\begin{theorem} \label{4.4}
Let $\Lieg$ be a Leibniz algebra such that $\gamma_2^{\Lie}(\Lieg) \neq 0$ and $\Gamma^{\Lie}(\Lieg) = \K. \id$, then $\Gamma^{\Lie}(A \otimes \Lieg) = \Gamma(A) \otimes \Gamma^{\Lie}(\Lieg) \cong A$.
\end{theorem}
\begin{proof}
By Lemma~\ref{4.1} it is enough to prove that $\Gamma^{\Lie}(A \otimes \Lieg) \subseteq \Gamma(A) \otimes \Gamma^{\Lie}(\Lieg) \cong A$.
Assume that $\{ A_i \}$ is a basis of $A$. Then for any $\varphi \in \Gamma^{\Lie}(A \otimes \Lieg), a \in A$, there exists a set of linear transformations $\{ \eta_i(a, -) \colon {\Lieg} \to {\Lieg} \}$ such that for all $g \in {\Lieg}$,
\begin{equation} \label{eta}
	\varphi(a \otimes g) = \sum_i A_i \otimes \eta_i(a, g)
\end{equation}
where the sum has only a finite number of non-zero summands.

We claim that $\eta_i(a, -)$ are elements of $\Gamma^{\Lie}(\Lieg)$, for all $i$. Indeed, on one hand
\begin{align*}
\varphi([a_1 \otimes g_1, 1 \otimes g_2]_{\Lie}) &= [\varphi(a_1 \otimes g_1), 1 \otimes g_2]_{\Lie} \\
{}&= \left[ \sum_i A_i \otimes \eta_i(a_1, g_1), 1 \otimes g_2 \right]_{\Lie}\\
{}&= \sum_i A_i \otimes \left[ \eta_i(a_1, g_1), g_2 \right]_{\Lie},
\end{align*}
 and on the other hand
\begin{align*}
\varphi([a_1 \otimes g_1, 1 \otimes g_2]_{\Lie}) &= \varphi(a_1 \otimes [g_1, g_2]_{\Lie}) \\
{}&= \sum_i A_i \otimes \eta_i(a_1, [g_1, g_2]_{\Lie}).
\end{align*}
Therefore $\left[ \eta_i(a_1, g_1), g_2 \right]_{\Lie} = \eta_i(a_1, [g_1, g_2]_{\Lie})$. A similar computation shows the other equality.

Since $\Gamma^{\Lie}(\Lieg) = \K. \id$, then $\eta_i(a,g)=\lambda_i(a).g$, $\lambda_i(a) \in \K$. By identity~\eqref{eta}, for any~$\varphi \in \Gamma^{\Lie}(A \otimes \Lieg)$, $a \in A$, there exists a finite set of indices $J$ such that if $i \notin J$, then $\eta_i(a, -)=0$. Thus we have that for all $g \in {\Lieg}$
\[
	\varphi(a \otimes g) = \sum_i A_i \otimes \eta_i(a, g) = \sum_{i \in J} A_i \otimes \lambda_i(a).g = \sum_{i \in J} \lambda_i(a) (A_i \otimes g).
\]
 Let $\rho \colon A \to A$ be the map defined by $\rho(a) = \sum_{i \in J} \lambda_i(a) A_i$, for all $a \in A$. Then,~$\varphi(a \otimes g) = \rho(a) \otimes g$.
Since
\begin{align*}
\varphi([a_1 \otimes g_1, 1 \otimes g_2]_{\Lie}) &= \varphi(a_1 \otimes [g_1, g_2]_{\Lie}) = \rho(a_1) \otimes [g_1, g_2]_{\Lie},\\
[a_1 \otimes g_1, \varphi(1 \otimes g_2)]_{\Lie} &= [a_1 \otimes g_1, \rho(1) \otimes g_2]_{\Lie} = \rho(1) a_1 \otimes [g_1, g_2]_{\Lie},
\end{align*}
then we get $\rho(a_1) \otimes [g_1, g_2]_{\Lie} = \rho(1) a_1 \otimes [g_1, g_2]_{\Lie}$. Moreover, since $\gamma_2^{\Lie}({\Lieg}) \neq 0$, we have that $\rho(a_1) = \rho(1) a_1$, $a_1 \in A$. Consequently,
\[
\varphi(a \otimes g) = (\rho(1) a) \otimes g = (\rho(1) \otimes \id_{\Lieg})(a \otimes g),
\]
for all $a \in A, g \in {\Lieg}$. This means that $\varphi$ belongs to $\Gamma(A) \otimes \K.\id = \Gamma(A) \otimes \Gamma^{\Lie}(\Lieg)$.
\end{proof}

\begin{example}
Example~\ref{derivation}~(b) satisfies the requirements of Theorem~\ref{4.4}.
\end{example}

\begin{remark}
Theorem~\ref{4.4} does not hold if $A$ is not a unital algebra, as the following example shows.
Let $\Lieg$ be the complex five-dimensional Leibniz algebra with the basis $\{ a_1, a_2, a_3, a_4, a_5 \}$ and bracket operation
\begin{align*}
[a_2, a_1] &= -a_3, & [a_1, a_2] &= a_3, & [a_1, a_3] &= -2a_1,\\
[a_3, a_1] &= 2a_1,& [a_3, a_2] &= -2a_2,& [a_2, a_3] &= 2a_2,\\
[a_5, a_1] &= a_4,& [a_4, a_2] &= a_5,& [a_4, a_3] &= -a_4,\\
&&&& [a_5, a_3] &= a_5,
\end{align*}
and zero elsewhere \cite[Example 3.2]{Om}. It can be checked that $\gamma_2^{\Lie}({\Lieg}) = \langle \{ a_4, a_5 \} \rangle$ and $\Gamma^{\Lie}(\Lieg) = \K. \id$.

Let $\mathbb{C}[t]$ be the polynomial ring in the variable $t$ with coefficients in the field of the complex numbers. Let $B = t^m \mathbb{C}[t]$, $m > 0$, be the subalgebra of $\mathbb{C}[t]$, which does not contain the unit element. It is easy to show that $f(t).\id_B \otimes \id_{\Lieg} \in \Gamma^{\Lie}(B \otimes \Lieg)$, for any $f(t) \in \mathbb{C}[t]$. This fact implies that $\Gamma^{\Lie}(B \otimes \Lieg) \ncong B$, since the elements~$f(t).\id_B \otimes \id_{\Lieg}$ are identified with $f(t) \in \mathbb{C}[t]$, which is not an element of B.
\end{remark}

\begin{definition}
A linear map $\phi \in \Gamma^{\Lie}(A \otimes \Lieg)$ is said to have a \emph{finite $\Lieg$-image} if, for any $a \in A$ there exist finitely many $a_1, \dots, a_n \in A$ such that
\[
	\phi(\K a \otimes {\Lieg}) \subseteq (\K a_1 \otimes {\Lieg}) + \dots + (\K a_n \otimes {\Lieg}).
\]
\end{definition}

From now on we denote by $\Gamma(A) \widetilde{\otimes} \Gamma^{\Lie}(\Lieg)$ the $\K$-span of all endomorphisms $f \widetilde{\otimes} \varphi$ of~${A \otimes {\Lieg}}$. Due to Lemma~\ref{4.1}, we know that $\Gamma(A) \widetilde{\otimes} \Gamma^{\Lie}(\Lieg) \subseteq \Gamma^{\Lie}(A \otimes {\Lieg})$. Moreover,
\[
	\Gamma(A) \widetilde{\otimes} \Gamma^{\Lie}(\Lieg) \subseteq \{ \phi \in \Gamma^{\Lie}(A \otimes \Lieg) \mid \phi ~\text{has finite {\Lieg}-image} \}.
\]

\begin{lemma} \label{4.8}
Let $\{ a_i \}_{i \in I}$ be a basis of $A$ and $\varphi \in \Gamma^{\Lie}(A \otimes \Lieg)$. Let $\varphi_i \colon {\Lieg} \to {\Lieg}$ be the linear maps defined by $\varphi(1 \otimes g)= \sum_{i \in I} a_i \otimes \varphi_i(g)$. Then $\varphi_i \in \Gamma^{\Lie}(\Lieg)$.
\end{lemma}
\begin{proof}
\begin{align*}
\varphi(1 \otimes [g_1, g_2]_{\Lie}) &= \sum_{i \in I} a_i \otimes \varphi_i([g_1, g_2]_{\Lie}),\\
\varphi(1 \otimes [g_1, g_2]_{\Lie}) &= \varphi([1 \otimes g_1, 1 \otimes g_2]_{\Lie}) = [\varphi(1 \otimes g_1), 1 \otimes g_2]_{\Lie}\\
 &=[\sum_{i \in I} a_i \otimes \varphi_i(g_1), 1 \otimes g_2]_{\Lie} = \sum_{i \in I} a_i \otimes [\varphi_i(g_1), g_2]_{\Lie},\\
 \varphi(1 \otimes [g_1, g_2]_{\Lie}) &= \sum_{i \in I} a_i \otimes [g_1, \varphi_i(g_2)]_{\Lie}.
\end{align*}
Hence $\varphi_i([g_1, g_2]_{\Lie}) = [\varphi_i(g_1), g_2]_{\Lie} = [g_1, \varphi_i(g_2)]_{\Lie}$, for all $g_1, g_2 \in {\Lieg}$, and therefore~$\varphi_i \in \Gamma^{\Lie}(\Lieg)$.
\end{proof}

\begin{proposition}
	Let $\Lieg$ be  Leibniz algebra. Then,
	\begin{enumerate}
		\item[(i)] $\gamma_2^{\Lie}(A \otimes {\Lieg}) = A \otimes \gamma_2^{\Lie}(\Lieg)$.
		
		\item [(ii)]If $\gamma_2^{\Lie}(\Lieg)$ is finite generated as an $\Gamma^{\Lie}(\Lieg)$-module, then  every $\varphi \in \Gamma^{\Lie}(A \otimes  \gamma_2^{\Lie}(\Lieg))$ has finite $\gamma_2^{\Lie}(\Lieg)$-image.
		
		\item[(iii)] If $\gamma_2^{\Lie}({\Lieg}) \neq 0$ and $\Gamma^{\Lie}({\Lieg}) = \K.{\id}_{\Lieg}$, then $\varphi \in \Gamma^{\Lie}(A \otimes {\Lieg})$ has finite {\Lieg}-image.
	\end{enumerate}
\end{proposition}
\begin{proof}
	(i) For any $a \otimes g \in A \otimes \gamma_2^{\Lie}({\Lieg})$, we have
	\[
	a \otimes g = a.1 \otimes \sum_i k_i. [g_{i1}, g_{i2}]_{\Lie} = \sum_i k_i [a \otimes g_{i1}, 1 \otimes g_{i2}]_{\Lie} \in \gamma_2^{\Lie}(A \otimes {\Lieg}).
	\]
	(ii) Let $x_1, \dots, x_n \in \gamma_2^{\Lie}(\Lieg)$ such that  $\gamma_2^{\Lie}(\Lieg) = \psi_{\Lieg} x_1 + \cdots + \psi_{\Lieg} x_n$, where $\psi_{\Lieg} \in \Gamma^{\Lie}(\gamma_2^{\Lie}(\Lieg))$. Let  $\varphi \in  \Gamma^{\Lie}(A \otimes {\Lieg})$ and~$a \in A$. There exist finite families $\{ x_{ij} \}_{j \in J}$ in  $\gamma_2^{\Lie}(\Lieg)$ and $\{ a_{ij} \}_{j \in J}$ in A, such that~${\varphi(a \otimes x_i) = \sum_j a_{ij} \otimes x_{ij}}$, for $1 \leq i \leq n$. Hence, by Lemma \ref{4.1}, statement (i) and Proposition \ref{3.6}, we have
	\begin{align*}
		\varphi(a \otimes g) &=  \varphi \left( a \otimes \sum_i \psi_{\Lieg} x_i \right) = \sum_i \varphi  \left( a \otimes \psi_{\Lieg} x_i \right) \\
		{}&= \sum_i \varphi \circ \big(\id \otimes \psi_{\Lieg} \big) \left( a \otimes x_i \right) = \sum_i \big(\id \otimes \psi_{\Lieg} \big) \circ \varphi \left( a \otimes x_i \right) \\
		{}&= \sum_i \sum_j \big(\id \otimes \psi_{\Lieg} \big) \left( a_{ij} \otimes x_{ij} \right) \subseteq \sum_i \sum_j \left( a_{ij} \otimes {\Lieg} \right).
	\end{align*}
	
	(iii) Assume that $\Gamma^{\Lie}({\Lieg}) = \K.{\id}_{\Lieg}$ and $\varphi \in \Gamma^{\Lie}(A \otimes {\Lieg})$, then by Theorem~\ref{4.4} we have:
	\[
	\varphi(1 \otimes g) = \sum_{i \in I} a_i \otimes \varphi_i(g) = \sum_{i \in I} a_i \otimes (k_i \id)(g) = \sum_{i \in I} a_i \otimes (k_i g)
	\]
	Fixed $g \in {\Lieg}$, then almost all $k_i g = 0$, and almost $k_i = 0$, which in turn means that $\varphi$ has finite $\Lieg$-image.
\end{proof}

\begin{theorem}
Let $\Lieg$ be a Leibniz algebra such that $\gamma_2^{\Lie}({\Lieg}) \neq 0$, $\Gamma^{\Lie}({\Lieg}) = \K.\id$ and $R = \K[x_1, \dots, x_n]$. Then $\Gamma^{\Lie}(R \otimes {\Lieg}) = R \otimes \Gamma^{\Lie}({\Lieg})$.
\end{theorem}
\begin{proof}
The inclusion $R \otimes \Gamma^{\Lie}({\Lieg}) \subseteq \Gamma^{\Lie}(R \otimes {\Lieg})$ is provided by Lemma~\ref{4.1}.

Let us consider $\{ m_i \}$ a basis of $R$, $\varphi \in \Gamma^{\Lie}(R \otimes {\Lieg})$, a polynomial $p \in R$ and~$g_1, g_2 \in {\Lieg}$. Let us denote by $\eta_i(p, -)$ the suitable maps in ${\End}_{\K}({\Lieg})$ such that~$\varphi(p \otimes g) = \sum_i m_i \otimes \eta_i(p, g)$, as in~\eqref{eta}.
Then,
\begin{align*}
\varphi \left( [p \otimes g_1, 1 \otimes g_2]_{\Lie} \right)&= \left[\varphi \left( p \otimes g_1 \right), 1 \otimes g_2 \right]_{\Lie} = \left[\sum_i m_i \otimes \eta_i(p,g_1), 1 \otimes g_2 \right]_{\Lie}\\
{}&= \sum_i m_i \otimes \left[\eta_i(p,g_1), g_2 \right]_{\Lie},
\end{align*}
but
\[
	\varphi \left( [p \otimes g_1, 1 \otimes g_2]_{\Lie} \right) = \varphi \left( p \otimes [g_1, g_2]_{\Lie} \right) = \sum_i m_i \otimes \eta_i(p, [g_1, g_2]_{\Lie}).
\]
Consequently, for every $i$ and $p$, we have $\left[\eta_i(p,g_1), g_2 \right]_{\Lie} = \eta_i(p, [g_1, g_2]_{\Lie})$. With a similar computation, we have that \[
	\left[\eta_i(p,g_1), g_2 \right]_{\Lie} = [g_1, \eta_i(p, g_2)]_{\Lie},
\]
i.e., $\eta_i(p, -) \in \Gamma^{\Lie}({\Lieg}) = \K.\id$. Therefore, $\eta_i(p, g) = k_i(p)g$, for all $g \in {\Lieg}$ and suitable scalars $k_i(p)$. Hence, $\varphi(p \otimes g) = \sum_i m_i \otimes k_i(p)g = \sum_i k_i(p) m_i \otimes g$, which is an element in $R \otimes {\Lieg}$ for every $p$, so $k_i(p)$ is non-zero for a finite number of $i$ (it is enough to take any $g \in {\Lieg}$, $g \neq 0$, to see this), that is
\[
	\varphi(p \otimes g) = \sum_i^n m_i \otimes k_i(p) g = \sum_i^n k_i(p) m_i \otimes g.
\]
Then the map $\rho \colon p \mapsto \displaystyle \sum_i^n k_i(p) m_i$ is well-defined. Thus we have $\varphi(p \otimes g) = \rho(p) \otimes p$, and hence,
\begin{align*}
	[\varphi(p \otimes g), 1 \otimes g']_{\Lie} &= [\rho(p) \otimes g, 1 \otimes g']_{\Lie} = \rho(p) \otimes [g, g']_{\Lie}, \\
	[p \otimes g, \varphi(1 \otimes g')]_{\Lie} &= [p \otimes g, \rho(1) \otimes g']_{\Lie} = p \rho(1) \otimes [g, g']_{\Lie}.
\end{align*}
Choosing $g, g' \in {\Lieg}$ such that $[g, g']_{\Lie} \neq 0$, we conclude that ${\rho(p)=p \rho(1)}$, for all~$p \in R$. This means that $\rho$ is determined by the action on $1$, and therefore~$\rho \in R$.
Thus,
\[
\varphi(p \otimes g) = \rho(p) \otimes g = (\rho \otimes \id)(p \otimes g),
\]
which means exactly that $\varphi \in R \otimes \id \cong R \otimes \Gamma^{\Lie}({\Lieg})$.
\end{proof}



\end{document}